\documentclass{amsart}
\usepackage{amsfonts, color}
\usepackage{latexsym}
\usepackage{amssymb}
\usepackage{amsmath}

\newcommand{\R}{\mathbb R}
\newcommand{\N}{\mathbb N}

\newtheorem{thm}{Theorem}[section]
\newtheorem{cor}{Corollary}[section]
\newtheorem{lemma}{Lemma}[section]

\newtheorem{proposition}{Proposition}[section]

\theoremstyle{remark}
\newtheorem*{rmk}{Remark}

\numberwithin{equation}{section}

\begin{document}


\title{Reverse Loomis-Whitney inequalities via isotropicity}

\author[D.\,Alonso]{David Alonso-Guti\'errez}
\address{\'Area de an\'alisis matem\'atico, Departamento de matem\'aticas, Facultad de Ciencias, Universidad de Zaragoza, Pedro Cerbuna 12, 50009 Zaragoza (Spain), IUMA}
\email[(David Alonso)]{alonsod@unizar.es}

\author[S. Brazitikos]{Silouanos Brazitikos}
\address{School of Mathematics and Maxwell Institute for Mathematical
	Sciences, University of Edinburgh, JCMB, Peter Guthrie Tait Road King's Buildings,
	Mayfield Road, Edinburgh, EH9 3FD, Scotland.}
\email[(Silouanos Brazitikos)]{silouanos.brazitikos@ed.ac.uk}
\thanks{The first author is partially supported by MINECO Project MTM2016-77710-P and DGA E-26\_17R}
\begin{abstract}
Given a centered convex body $K\subseteq\R^n$, we study the optimal value of the constant $\tilde{\Lambda}(K)$ such that there exists an orthonormal basis $\{w_i\}_{i=1}^n$ for which the following reverse dual Loomis-Whitney inequality holds:
$$
|K|^{n-1}\leqslant \tilde{\Lambda}(K)\prod_{i=1}^n|K\cap w_i^\perp|.
$$
We prove that $\tilde{\Lambda}(K)\leqslant(CL_K)^n$ for some absolute $C>1$ and that this estimate in terms of $L_K$, the isotropic constant of $K$, is asymptotically sharp in the sense that there exists another absolute constant $c>1$ and a convex body $K$ such that $(cL_K)^n\leqslant\tilde{\Lambda}(K)\leqslant(CL_K)^n$. We also prove more general reverse dual Loomis-Whitney inequalities as well as reverse restricted versions of Loomis-Whitney and dual Loomis-Whitney inequalities.
\end{abstract}

\date{\today}
\maketitle
\section{Introduction and notation}

The classical Loomis-Whitney inequality \cite{LW} states that given a fixed orthonormal basis $\{e_i\}_{i=1}^n$, for any convex body $K\subseteq\R^n$  we have that
\begin{equation}\label{eq:ClassicalLoomisWhitney}
|K|\leqslant\prod_{i=1}^n|P_{e_i^\perp}K|^\frac{1}{n-1},
\end{equation}
where $|\cdot|$ denotes the volume (i.e., the Lebesgue measure) in the corresponding subspace and, for any $k$-dimensional linear subspace $H\in G_{n,k}$, $P_H$ denotes the orthogonal projection onto $H$. Convex body is a compact convex set with non-empty interior and the set of all convex bodies $K\subseteq\R^n$ will be denoted by $\mathcal{K}^n$. 
The barycentre of a convex body $K\in\mathbb{R}^n$ is the vector 
$$\mathrm{bar}(K)=\frac{1}{|K|}\int_{K}x\, dx.$$ We call $K$ centered if $\mathrm{bar}(K)=0$ and  the set of all centered convex bodies will be denoted by $\mathcal{K}_c^n$. Finally, the set of all centrally symmetric convex bodies will be denoted by $\mathcal{K}_0^n$.

In \cite{M}, Meyer proved the following dual inequality: For any convex body $K\subseteq\R^n$
\begin{equation}\label{eq:MeyersInequality}
|K|\geqslant\frac{(n!)^\frac{1}{n-1}}{n^\frac{n}{n-1}}\prod_{i=1}^n|K\cap e_i^\perp|^\frac{1}{n-1}.
\end{equation}

In \cite{CGG}, Campi, Gritzmann and Gronchi considered the following problem. Given any convex body $K\subseteq\R^n$, find the largest constant $\Lambda(K)$ such that there exists an orthonormal basis $\{w_i\}_{i=1}^n$ for which the following inequality, reverse to the classical Loomis-Whitney inequality \eqref{eq:ClassicalLoomisWhitney}, holds:
\begin{equation}\label{eq:ReverseLoomisWhitneyCGG}
	|K|^{n-1}\geqslant\Lambda(K)\prod_{i=1}^n|P_{w_i^\perp}K|.
\end{equation}
In the aforementioned paper the authors were interested in finding the value of $\Lambda(n):=\inf_{K\in\mathcal{K}^n}\Lambda(K)$. They found the exact value of this constant in the planar case and gave a lower bound for its value in any dimension. Subsequently, in \cite{KSZ}, Koldobsky, Saroglou and Zvavitch gave the right asymptotic estimate for the value of the constant, of the order $\Lambda(n)^\frac{1}{n}\simeq n^{-\frac{1}{2}}$. Here, and through the whole text, the notation $a\simeq b$ is used to denote the existence of two absolute constants $c_1,c_2>0$ such that $c_1b\leqslant a\leqslant c_2b$.

In \cite{FHL}, Feng, Huang and Li considered the dual problem. Given any centered convex body $K\subseteq\R^n$, find the best constant $\tilde{\Lambda}(K)$ such that there exists an orthonormal basis $\{w_i\}_{i=1}^n$ for which the following inequality, reverse to the dual Loomis-Whitney inequality \eqref{eq:MeyersInequality} holds:
\begin{equation}\label{eq:ReverseDualLoomisWhitney}
	|K|^{n-1}\leqslant\tilde{\Lambda}(K)\prod_{i=1}^n|K\cap w_i^\perp|.
\end{equation}
They proved that if $K$ is a centrally symmetric convex body in $\R^n$ then $\tilde{\Lambda}(K)\leqslant((n-1)!)^n$. In other words, given a centered convex body $K\subseteq\R^n$, we are interested in the value of
\begin{equation}\label{eq:ReverseDualLoomisWhitney2}
\tilde{\Lambda}(K)=\min\frac{|K|^{n-1}}{\prod_{i=1}^n|K\cap w_i^\perp|},
\end{equation}
where the minimum is taken over all the orthogonal bases $\{w_i\}_{i=1}^n$ of $\R^n$. Moreover, we define
$$\tilde{\Lambda}(n)=\sup_{K\in\mathcal{K}_c^n}\tilde{\Lambda}(K)\quad \textrm{and}\quad \tilde{\Lambda}_0(n)=\sup_{K\in\mathcal{K}_0^n}\tilde{\Lambda}(K),$$
where the supremum is taken over all centered convex bodies $K$ in $\mathbb{R}^n$ and over all centrally symmetric convex bodies respectively.

In this note, we describe the exact asymptotic behavior of $\tilde{\Lambda}(n)$ given by the following theorem. The precise definition of $L_K$, the isotropic constant of $K$, will be given in Section \ref{sec:Preliminaries}.
\begin{thm}\label{thm_Lambda_n_1}
	For every centered convex body $K\in\mathcal{K}_c^n$, we have that
	$$\tilde{\Lambda}(K)\leqslant \left(2\sqrt{3}L_K\right)^n.$$
Furthermore,
	$$\left(\sqrt{2}L_n\right)^n\leqslant \tilde{\Lambda}(n)\leqslant \left(2\sqrt{3}L_n\right)^n,$$
	where $L_n=\max_{K\in\mathcal{K}^n}{L_K},$ is the maximal isotropic constant.
\end{thm}

\begin{rmk}
Notice that the best known general upper bound for the isotropic constant (see section \ref{sec:Preliminaries}) gives an estimate
$\tilde{\Lambda}(n)\leqslant (Cn^\frac{1}{4})^n$, improving the estimate $\tilde{\Lambda}(n)\leqslant((n-1)!)^n$. Moreover, if we assume that the hyperplane conjecture is true, we have that $\tilde{\Lambda}(n)^\frac{1}{n}\simeq 1$.
\end{rmk}

As a consequence we obtain that for every centrally symmetric planar convex body $K\in\mathcal{K}_0^2$, we have that $\tilde{\Lambda}(K)\leq 1$. This inequality was proved in \cite{FHL}, where the equality cases were claimed to be characterized. Unfortunately, such characterization is not correct and, while it is true that $\tilde{\Lambda}_0(2)=1$, the equality cannot be attained for any convex body (see Section \ref{sec:PlanarCase}).

Moreover, we prove the following general reverse inequality for sections of arbitrary dimension. Before stating the theorem, we need a more general definition for $\tilde{\Lambda}(K)$ and $\tilde{\Lambda}(n)$. Let $m\geqslant 1$ and let $\mathcal{S}=(S_1,\dots, S_m)$ be a uniform cover of $[n]:=\{1,\dots, n\}$ with weights $(p_1,\dots, p_m)$, that is  $S_j\subseteq[n]$ for every $1\leqslant j\leqslant m$ and for every $1\leqslant i\leqslant n$
$$
\sum_{j=1}^mp_j\chi_{S_j}(i)=1.
$$
For any basis $\{w_i\}_{i=1}^n$ of $\R^n$, let $H_j={\rm span}\{w_k\,:\,k\in S_j\}$, $d_j={\rm dim}H_j=|S_j|$, and $p=\sum_{j=1}^mp_j$.

For every $\mathcal{S}$, we are interested in the value of
$$\tilde{\Lambda}_{\mathcal{S}}(K)=\min\frac{|K|^{p-1}}{\prod_{i=1}^n|K\cap H_j^\perp|^{p_j},}$$
where the minimum is taken over all the orthogonal bases $\{w_i\}_{i=1}^n$ of $\R^n$.
Moreover, let $$\tilde{\Lambda}_{\mathcal{S}}(n)=\sup_{K\in\mathcal{K}_c^n}\tilde{\Lambda}_{\mathcal{S}}(K),$$
where the supremum is taken over all centered convex bodies $K$ in $\mathbb{R}^n$. Then, we have the following

\begin{thm}\label{thm:ReverseDualLoomisWhitneyGeneral}
There exists an absolute constant $C>0$, such that for every centered convex body $K\in\mathcal{K}_c^n$ for any uniform cover $\mathcal{S}=(S_1,\dots, S_m)$ of $[n]$ with weights $(p_1,\dots, p_m)$, we have that
	$$
\tilde{\Lambda}_{\mathcal{S}}(K) \leqslant (CL_K)^n.
	$$
Furthermore, there exist absolute constants $c,C$ such that
	$$
\frac{(cL_n)^n}{\prod_{j=1}^mL_{d_j}^{p_jd_j}}\leqslant \tilde{\Lambda}_{\mathcal{S}}(n) \leqslant (CL_n)^n,
	$$
where $L_d=\max_{K\in\mathcal{K}^d}L_K$ is the maximal isotropic constant in $\R^d$.
\end{thm}

\begin{rmk}
Again, if we assume that the hyperplane conjecture is true we have $\tilde{\Lambda}_{\mathcal{S}}(n)^\frac{1}{n}\simeq 1$.
\end{rmk}
In \cite{BGL}, the following restricted Loomis-Whitney inequality was obtained; if $S\subseteq[n]$ has  cardinality $|S|=d$ and $(S_1,\dots, S_m)$ form a uniform cover of $S$ with the same weights $(\frac{1}{k},\dots,\frac{1}{k})$,  where $m>k$, then for every convex body $K\subseteq\R^n$ and any orthogonal basis $\{e_i\}_{i=1}^n$
$$
|P_{H^\perp}K||K|^{\frac{m}{k}-1}\leqslant\frac{{{n-\frac{kd}{m}}\choose{n-d}}^\frac{m}{k}}{{{n\choose d}}^{\frac{m}{k}-1}}\prod_{j=1}^m|P_{H_j^\perp}K|^\frac{1}{k}.
$$
where $H_j=\textrm{span}\{e_k\,:\,k\in S_j\}$ and $H=\textrm{span}\{e_k\,:\,k\in S\}$. In particular, for every convex body $K\subseteq\R^n$ and any $d$-dimensional subspace $H\in G_{n,d}$, we have that for any orthogonal basis $\{e_i\}_{i=1}^d$ of $H$
\begin{equation}\label{eq:RestictedLoomisWhitney}
|P_{H^\perp}K||K|^{d-1}\leqslant\frac{{{n-1}\choose{n-d}}^d}{{{n\choose d}}^{d-1}}\prod_{j=1}^d|P_{e_j^\perp}K|.
\end{equation}
 Dual restricted inequalities were also proved in \cite{BGL}. We will also consider the problem of finding reverse restricted Loomis-Whitney inequalities and restricted dual Loomis-Whitney inequalities. We will prove the following two results:

\begin{thm}\label{thm:ReverseLocalLoomisWhitney}
	Let $K\in\mathcal{K}^n$ be a convex body and let $2\leqslant d\leqslant n-1$. For any $H\in G_{n,d}$ there exists an orthonormal basis $\{w_j\}_{j=1}^d$ of $H$ such that if we denote $H={\rm span}\{w_1,\dots,w_d\}$ then we have that
	$$
	|P_{H^\perp}K||K|^{d-1}\geqslant\frac{{{n+d}\choose {n}}}{(2n)^d}\prod_{i=1}^d|P_{w_i^\perp}K|.
	$$
\end{thm}
\begin{rmk}
	Notice that if $d=2$ then the constant in Theorem \ref{thm:ReverseLocalLoomisWhitney} and the constant in equation \eqref{eq:RestictedLoomisWhitney} are of the same order.
\end{rmk}
\begin{thm}\label{thm:ReverseLocalDualLoomis-Whitney}
	There exists an absolute constant $C$ such that for every centered convex body $K\in\mathcal{K}_c^n$ and every $H\in G_{n,d}$ there exists an orthonormal basis $\{w_j\}_{j=1}^d$ of $H$ such that
	$$
	|K||K\cap H^\perp|^{d-1}\leq C^{d(d-1)}d^\frac{d}{2}\prod_{j=1}^d|K\cap (H^\perp\oplus\langle w_j\rangle)|.
	$$
\end{thm}

The paper is organized as follows: In Section \ref{sec:Preliminaries} we provide the preliminary definitions and results that we use in order to prove our results. In Section \ref{sec:ProofReverseDualLoomisWhitney} we prove the reverse dual Loomis-Whitney inequalities given by Theorem \ref{thm_Lambda_n_1} and Theorem \ref{thm:ReverseDualLoomisWhitneyGeneral}. In Section \ref{sec:PlanarCase} we study the situation in the centrally symmetric planar case. Finally, in Section \ref{sec:RestrictedVersions} we prove the restricted versions provided in Theorems \ref{thm:ReverseLocalLoomisWhitney} and \ref{thm:ReverseLocalDualLoomis-Whitney}.

\section{Preliminaries}\label{sec:Preliminaries}

A convex body $K\in\mathcal{K}^n$ is called isotropic if $|K|=1$, $K$ is centered, and for every $\theta\in S^{n-1}$
$$
\int_K\langle x,\theta\rangle^2 dx=L_K^2,
$$
where $L_K$ is a constant depending on $K$, but not on $\theta$, which is called the isotropic constant of $K$. Given any convex body $K\subseteq\R^n$ there exists an  affine map $a+T$, with $a\in\R^n$ and $T\in GL(n)$ (unique up to orthogonal transformations), such that $a+TK$ is isotropic. The isotropic constant of $K$ is then defined as the isotropic constant of any of its isotropic images. Such an affine map is the solution of a minimization problem, which allows to alternatively define $L_K$ in the following way
$$
nL_K^2=\min\left\{\frac{1}{|K|^{1+\frac{2}{n}}}\int_{a+TK}|x|^2\,:\,a\in\R^n,T\in GL(n)\right\}.
$$
It is well known that the Euclidean ball $B_2^n$ is the $n$-dimensional convex body with the smallest isotropic constant and, as a consequence, there exists an absolute constant $c>0$ such that $L_K\geqslant c$ for every convex body $K\subseteq\R^n$ and any $n\in\N$ (see, for instance, \cite[Proposition 3.3.1]{BGVV}. However, it is still a major open problem (known as the slicing problem) whether there exists an absolute constant $C>0$ such that $L_n:=\max_{K\in\mathcal{K}^n}L_K\leqslant C$. This question was posed by Bourgain, who proved the upper bound $L_n\leqslant Cn^\frac{1}{4}\log n$ in \cite{Bo}. This was improved to $L_n\leqslant Cn^\frac{1}{4}$ by Klartag in \cite{K} and it is the currently best known bound. In the planar case, it is known (see \cite[Theorem 3.5.7]{BGVV} and the results in \cite{S}) that $L_2=L_{\Delta^2}=\frac{1}{\sqrt{6}\sqrt[4]{3}}$. If we restrict ourselves to centrally symmetric convex bodies and denote $L_{n,0}:=\max_{K\in\mathcal{K}_0^n}L_K$, then $L_{2,0}=L_{B_\infty^2}=\frac{1}{\sqrt{12}}$. Here $\Delta^n$ denotes the $n$-dimensional regular simplex and $B_\infty^n$ denotes the $n$-dimensional cube. These (and their affine images) are the only convex bodies on which the maximums in $\mathcal{K}^n$ (and in $\mathcal{K}_0^n$) are attained.

Given a centered convex body $K\in\mathcal{K}_c^n$ with $|K|=1$ and $p>1$, its $L_p$-centroid body $Z_p(K)$ is defined by
$$
h_{Z_p(K)}(y)=\left(\int_{K}|\langle x,y\rangle|^pdx\right)^\frac{1}{p},\quad y\in\R^n,
$$
where for any convex body $L\in\mathcal{K}^n$, $h_L(y)=\max\{\langle x,y\rangle\,:\,x\in L\}$ is the support function of $L$. Notice that, by H\"older's inequality, if $1\leqslant p\leqslant q$ then $Z_p(K)\subseteq Z_q(K)$. Moreover, for any linear map $T\in SL(n)$, with $|\textrm{det}T|=1$, $Z_p(TK)=TZ_p(K)$, and that $K$ is isotropic if and only if $Z_2(K)=L_KB_2^n$. If $K$ is not isotropic and $|K|=1$ then $Z_2(K)$ is an ellipsoid whose volume is $|Z_2(K)|=L_K^n|B_2^n|$ (see, for instance \cite[Proposition 3.1.7]{BGVV}). In \cite{H}, Hensley proved that there exist two absolute constants $c_1,c_2$ such that for every centered convex body $K\in\mathcal{K}_c^n$ with $|K|=1$ and every $\theta\in S^{n-1}$
\begin{equation}\label{eq:Hensley}
\frac{c_1}{|K\cap \theta^\perp|}\leqslant h_{Z_2(K)}(\theta)\leqslant\frac{c_2}{|K\cap \theta^\perp|}.
\end{equation}
The value of these two constants are known to be (see \cite[Corollaries 2.5 and 2.7]{MP} and \cite[Theorem 3]{Fra}) $c_1=\frac{1}{2\sqrt{3}}$ and $c_2(n)=\frac{n}{\sqrt{2(n+1)(n+2)}}\leq\frac{1}{\sqrt{2}}$. Furthermore, there is equality in the left-hand side inequality if and only if $K$ is cylindrical in the direction $\theta$ (i.e., $K=K\cap\theta^\perp+[-x,x]$ for some $x\in\R^n$) and there is equality in the right hand-side inequality if and only if $K$ is a double cone in the direction $\theta$.

The latter equation shows that for any isotropic convex body and any $\theta\in S^{n-1}$
$$
|K\cap\theta^\perp|\simeq\frac{1}{L_K}.
$$
More generally, in  \cite[Proposition 3.11]{MP} (see also \cite[Proposition 5.1.15]{BGVV}) it was proved that for any isotropic convex body $K$ and any $d$-dimensional linear subspace $H\in G_{n,d}$, there exists a $d$-dimensional convex body $B(K,H)$ such that
\begin{equation}\label{eq:MilmaPajorHensley}
|K\cap H^\perp|^\frac{1}{d}\simeq\frac{L_{B(K,H)}}{L_K}.
\end{equation}

It was proved by Paouris (see \cite[Theorem 5.1.14]{BGVV}) that there exist two absolute constants $c_1,c_2$ such that for every centered convex body $K\in\mathcal{K}^n$ with $|K|=1$ and every $d$-dimensional linear subspace $H\in G_{n,d}$
\begin{equation}\label{eq:ProjectionsCentroidBodies}
c_1\leqslant|K\cap H^\perp|^\frac{1}{d}|P_HZ_d(K)|^\frac{1}{d}\leqslant c_2.
\end{equation}

Given a convex body $K\in\mathcal{K}^n$, its polar projection body $\Pi^*K$ is the closed unit ball of the norm given by
$$
\Vert x\Vert_{\Pi^*K}=|x||P_{x^\perp}K|,
$$
which is a centrally symmetric convex body. Equivalently, its radial function is given by $\rho_{\Pi^*K}(\theta)=\frac{1}{|P_{\theta^\perp}K|}$, where for every convex body $L\in\mathcal{K}^n$ containing the origin in its interior, its radial function is defined for every $\theta\in S^{n-1}$ by $\rho_{L}(\theta)=\max\{\lambda>0\,:\,\,\lambda \theta\in L\}$. It is well known that for any convex body $K\in\mathcal{K}^n$, the affinely invariant quantity $|K|^{n-1}|\Pi^*K|$ is maximized when $K$ is an ellipsoid and minimized when $K$ is a simplex (see \cite{P} and \cite{Z}). Thus, for every convex body $K\subseteq\R^n$
$$
\frac{{{2n}\choose {n}}}{n^n}\leq|K|^{n-1}|\Pi^*K|\leq\left(\frac{|B_2^n|}{|B_2^{n-1}|}\right)^n.
$$
In \cite[Proposition 5.2]{AGJ}, it was proved that for any convex body $K\in\mathcal{K}^n$ and any $d$-dimensional linear subspace $H\in G_{n,d}$
	\begin{equation}\label{eq:SectionsPolarProjectionBody}
	|K|^{d-1}|\Pi^*K\cap H|\geqslant\frac{{{n+d}\choose{n}}}{n^d|P_{H^\perp}K|}.
	\end{equation}
\section{Proof of the reverse dual Loomis Whitney inequality}\label{sec:ProofReverseDualLoomisWhitney}

We begin this section by proving Theorem \ref{thm_Lambda_n_1}

\begin{proof}[Proof of Theorem \ref{thm_Lambda_n_1}]
	Let $K$ be a centered convex body. We can assume without loss of generality that $|K|=1$. Let $Z_2(K)\subseteq\R^n$ be the ellipsoid whose support function is given by
	$$
	h_{Z_2(K)}(w)=\left(\int_K\langle x,w\rangle^2dx\right)^\frac{1}{2}
	$$
	for every $w\in S^{n-1}$. We have that $|Z_2(K)|=L_K^n|B_2^n|$. By \eqref{eq:Hensley} there exist two absolute constants $c_1=\frac{1}{2\sqrt{3}},c_2=\frac{1}{\sqrt{2}}$ such that for every centered convex body $K\subseteq\R^n$ with volume $1$ and every $w\in S^{n-1}$
	$$
	\frac{c_1}{|K\cap w^\perp|}\leqslant h_{Z_2(K)}(w)\leqslant\frac{c_2}{|K\cap w^\perp|}.
	$$
	Therefore, taking $\{w_i\}_{i=1}^n$ the orthonormal basis given by the principal axes of the ellipsoid $Z_2(K)$ we have
	$$
	\prod_{i=1}^n|K\cap w_i^\perp|\geqslant\frac{c_1^n}{\prod_{i=1}^nh_{Z_2(K)}(w_i)}=\frac{c_1^n|B_2^n|}{|Z_2(K)|}=\frac{c_1^n}{L_K^n},
	$$
	which proves that
$$
\tilde{\Lambda}(K)\leq(CL_K)^n
$$
with $C=\frac{1}{c_1}=2\sqrt{3}$. To conclude the proof of Theorem \ref{thm_Lambda_n_1} we first notice that from the above 
$$\tilde{\Lambda}(n)\leqslant (2\sqrt{3}L_n)^n.$$
On the other hand, if we consider an isotropic convex body with isotropic constant  $L_K=L_n$ we have that for every orthonormal basis $\{w_i\}$ of $\R^n$
	$$
	\frac{c_1^n}{L_K^n}\leqslant\prod_{i=1}^n|K\cap w_i^\perp|\leqslant\frac{c_2^n}{L_K^n},
	$$
and, since $L_K=L_n$,  $$\tilde{\Lambda}(n)\geqslant (cL_n)^n,$$
with $c=\frac{1}{c_2}=\sqrt{2}$. This concludes the proof.
\end{proof}
\begin{rmk}
The latter proof shows that for every isotropic convex body, $\tilde{\Lambda}(K)^\frac{1}{n}\simeq L_K$.
\end{rmk}

\bigskip

We now move to the general case.

\begin{proof}[Proof of Theorem \ref{thm:ReverseDualLoomisWhitneyGeneral}]
	 Let $m\geqslant 1$ and let $\mathcal{S}=(S_1,\dots, S_m)$ be a uniform cover of $[n]$ with weights $(p_1,\dots, p_m)$. Let $K$ be a centered convex body. We can assume without loss of generality that $|K|=1$. Let $\{w_i\}_{i=1}^n$ be the orthonormal basis given by the principal axes of the ellipsoid $Z_2(K)$, whose support function is given by
	 $$
	 h_{Z_2(K)}(w)=\left(\int_K\langle x,w\rangle^2dx\right)^\frac{1}{2}.
	 $$
	 Let $T\in GL(n)$ be the diagonal map with respect to the orthonormal basis $\{w_i\}_{i=1}^n$ given by $T(w_i)=\lambda_iw_i$ such that $TK$ is isotropic. By \eqref{eq:MilmaPajorHensley}, there exists an absolute constant $c_1$ such that for any $1\leqslant j\leqslant m$ there exists a $d_j$-dimensional convex body $B(K,H_j)$, depending on $K$ and $H_j=\textrm{span}\{w_k\,:\,k\in S_j\}$, verifying
\begin{eqnarray*}
	|K\cap H_j^\perp|&=&|T^{-1}T(K\cap H_j^\perp)|=\prod_{k\not\in S_j}\frac{1}{\lambda_k}|TK\cap H_j^\perp|\cr
&\geqslant&\left(\frac{c_1L_{B(K,H_j)}}{L_K}\right)^{d_j}\prod_{k\not\in S_j}\frac{1}{\lambda_k}.
\end{eqnarray*}
Note that $\displaystyle{\sum_{j=1}^m p_jd_j=n}$, the m-tuple $(S_1^c,\dots, S_m^c)$ forms a uniform cover of $[n]$ with weights $(p_1^\prime,\dots, p_m^\prime)$, where $p_i^\prime=\frac{p_i}{p-1}$, and $\prod_{i=1}^n\lambda_i=|T|=1$ since $|K|=|TK|=1$. Combining the above and calling $\displaystyle{p=\sum_{i=1}^n p_i}$ we get 
	\begin{align*}
		\prod_{j=1}^m|K\cap H_j^\perp|^{p_j}&\geqslant\left(\frac{c_1}{L_K}\right)^n\prod_{j=1}^m\left(L_{B(K,H_j)}\right)^{p_jd_j}\frac{1}{\prod_{k\not\in S_j}\lambda_k^{p_j}}\\
		&=\left(\frac{c_1}{L_K}\right)^n\frac{\prod_{j=1}^m\left(L_{B(K,H_j)}\right)^{p_jd_j}}{\prod_{i=1}^n\lambda_i^{\sum_{j=1}^mp_j\chi_{S_j^c(i)}}}\\
		&=\left(\frac{c_1}{L_K}\right)^n\frac{\prod_{j=1}^m\left(L_{B(K,H_j)}\right)^{p_jd_j}}{\prod_{i=1}^n\lambda_i^{p-1}}\\
		&=\frac{c_1^n\prod_{j=1}^m\left(L_{B(K,H_j)}\right)^{p_jd_j}}{L_K^n}.
	\end{align*}
	This means that 
$$
\tilde{\Lambda}_{\mathcal{S}}(K)\leqslant\frac{(L_K)^n}{c_1^n\prod_{j=1}^m(L_{B(K,H_j)})^{p_jd_j}}.
$$
Taking now the supremum over all orthonormal bases, and taking into account that there exists a universal constant $\tilde{c}>0$ bounding from below the isotropic constant of any convex body in any dimension, we get that
$$\tilde{\Lambda}_{\mathcal{S}}(K)\leqslant\max  \frac{(L_K)^n}{c_1^n\prod_{j=1}^m(L_{B(K,H_j)})^{p_jd_j}}\leqslant (CL_K)^n,$$
with $C=\frac{1}{\tilde{c}c_1}$.

	If $K$ is isotropic then, by \eqref{eq:MilmaPajorHensley}, there exists a universal constant $c_2$ such that for any orthonormal basis $\{w_i\}_{i=1}^n$ and any uniform cover $\mathcal{S}=(S_1,\dots, S_m)$ of $[n]$ with weights $(p_1,\dots, p_m)$, we have that for every $1\leqslant j\leqslant m$ the $d_j$-dimensional convex bodies $B(K,H_j)$ associated to $K$ and $H_j=\textrm{span}\{w_k\,:\,k\in S_j\}$ verifies
	$$
	|K\cap H_j^\perp|\leqslant\left(\frac{c_2L_{B(K,H_j)}}{L_K}\right)^{d_j},
	$$
	and then for any orthonormal basis $\{w_i\}_{i=1}^n$ and any uniform cover $\mathcal{S}=(S_1,\dots, S_m)$ of $[n]$ with weights $(p_1,\dots, p_m)$
	$$
	|K|^{p-1}\geqslant \frac{(L_K)^n}{c_2^n\prod_{j=1}^m (L_{B(K,H_j)})^{p_jd_j}}\prod_{j=1}^m|K\cap H_j^\perp|^{p_j}.
	$$
Therefore, taking $c=\frac{1}{c_2}$
$$
\tilde{\Lambda}_{\mathcal{S}}(K)\geqslant \min \frac{(cL_K)^n}{\prod_{j=1}^m (L_{B(K,H_j)})^{p_jd_j}}\geqslant \frac{(cL_K)^n}{\prod_{j=1}^m L_{d_j}^{p_jd_j}},
$$
where the minimum is taken over all the orthogonal basis $\{w_i\}_{i=1}^n$ in $\R^n$. Taking the convex body with maximal isotropic constant in $\R^n$, we get the reverse bound for $\tilde{\Lambda}_{\mathcal{S}}(n)$.
\end{proof}

\begin{rmk}
Notice that if $K$ is isotropic then one has that for any orthonormal basis $\{w_i\}_{i=1}^n$
$$
 \frac{(cL_K)^n}{\prod_{j=1}^m (L_{B(K,H_j)})^{p_jd_j}}\leqslant\frac{|K|^{p-1}}{\prod_{j=1}^m|K\cap H_j^\perp|^{p_j}}\leqslant \frac{(CL_K)^n}{\prod_{j=1}^m (L_{B(K,H_j)})^{p_jd_j}},
$$
where $c,C$ are absolute constants and so
$$
\tilde{\Lambda}_{\mathcal{S}}(K)^\frac{1}{n}\simeq \min \frac{L_K}{\prod_{j=1}^m (L_{B(K,H_j)})^\frac{p_jd_j}{n}},
$$
where the minimum is taken over all orthonormal bases $\{w_i\}_{i=1}^n$ in $\R^n$.
\end{rmk}

\section{The centrally symmetric planar case}\label{sec:PlanarCase}

In this section we will study the centrally symmetric planar case and prove the following:
\begin{proposition}\label{prop:PlanarSymmetricCase}
The value of $\tilde{\Lambda}_0(2)$ is
$$
\tilde{\Lambda}_0(2)=1.
$$
However, there exists no centrally symmetric planar convex body $K\in\mathcal{K}_0^2$ such that $\tilde{\Lambda}(K)=1$.
\end{proposition}
In order to prove the proposition we will make use of the following lemma, which shows that when $K$ is a centrally symmetric planar box, one of the two orthogonal vectors for which we obtain the minimum defining $\tilde{\Lambda}(K)$ has to be the direction of one of the diagonals.
\begin{lemma}\label{lem:PlanarBoxes}
Let $K\in\mathcal{K}_0^2$ be a centrally symmetric box (i.e. a centrally symmetric orthogonal parallelepiped) with $|K|=1$. Then
$$
\tilde{\Lambda}(K)=\frac{|K|}{|K\cap{w_1^\perp}||K\cap{w_2^\perp}|}=\frac{l^4}{l^4+1}<1,
$$
where $w_1$ is the direction of a diagonal of $K$ and $w_2$ is orthogonal to $w_1$ and $l\geq 1$ is the length of the largest side of $K$.
\end{lemma}

\begin{rmk}
If we do not assume $|K|=1$, then $l^2>1$ is the quotient of the largest side and the shortest side of the box.
\end{rmk}
\begin{proof}
We can assume without loss of generality that the sides of $K$ are parallel to the coordinate axes. Let $l$ denote the length of the vertical side of the box, which we can assume to be the longest one. Then $l\geqslant 1$ and
$$K=\textrm{conv}\left\{\left(\frac{1}{2l},\frac{l}{2}\right),\left(-\frac{1}{2l},\frac{l}{2}\right),\left(\frac{1}{2l},-\frac{l}{2}\right),\left(-\frac{1}{2l},-\frac{l}{2}\right)\right\}.$$
Let us take $w_2^\perp=\{(x,y)\in\R^2\,:\, y=ax, a\in\R\}$ a generic linear hyperplane and $w_1$ an orthogonal vector to $w_2$. Thus, $w_1^\perp=\left\{(x,y)\in\R^2\,:\,y=-\frac{1}{a}x\right\}$. Notice that if $a\in\left[l^2,\infty\right)$ then $w_2^\perp$ intersects with the boundary of $K$, $\partial K$, in the horizontal sides at the points $P_1=\left(\frac{l}{2a},\frac{l}{2}\right)$ and $-P_1$ and $w_1^\perp$ in the vertical sides at the points $P_2=\left(\frac{1}{2l},-\frac{1}{2al}\right)$ and $-P_2$, while if $a\in\left[\frac{1}{l^2},l^2\right]$ both $w_1^\perp, w_2^\perp$ intersect $\partial K$ in the vertical sides, being $w_2^\perp\cap\partial K$ the points $P_1^\prime=\left(\frac{1}{2l},\frac{a}{2l}\right)$ and $-P_1^\prime$, and $w_1^\perp\cap\partial K$ the points $P_2^\prime=\left(\frac{1}{2l},-\frac{1}{2al}\right)$ and $-P_2^\prime$.

Therefore, if $a\in\left[l^2,\infty\right)$, we have that
$$
|K\cap w_1^\perp||K\cap w_2^\perp|=1+\frac{1}{a^2}
$$
and if $a\in\left[\frac{1}{l^2},l^2\right]$
$$
|K\cap w_1^\perp||K\cap w_2^\perp|=\frac{1}{l^2}\left(a+\frac{1}{a}\right)
$$
we have that $|K\cap w_1^\perp||K\cap w_2^\perp|$ is maximized  in $a\in\left[\frac{1}{l^2},\infty\right)$ for the values $a=l^2$ and $a=\frac{1}{l^2}$, which correspond to the cases in which either $w_2^\perp$ or $w_1^\perp$ passes through one of the vertices of the box. If this is the case,
$$
|K\cap w_1^\perp||K\cap w_2^\perp|=\frac{l^4+1}{l^4}.
$$
Since $K$ is symmetric with respect to the coordinate axes, we have that for any $a\in\left(-\infty,\frac{1}{l^2}\right)$ there exists another pair of orthogonal lines $\tilde{w}_1^\perp,\tilde{w}_2^\perp$ described as before by a parameter $a_1\in\left[\frac{1}{l^2},\infty\right)$ for which
$$
|K\cap w_1^\perp||K\cap w_2^\perp|=|K\cap \tilde{w}_1^\perp||K\cap \tilde{w}_2^\perp|.
$$
Since in the case where $w_1,w_2$ are the coordinate vectors we have $|K\cap w_1^\perp||K\cap w_2^\perp|=1$, it follows that
$$
\max|K\cap{w_1^\perp}||K\cap{w_2^\perp}|=\frac{l^4+1}{l^4},
$$
where the maximum is taken over all the pairs of orthogonal vectors in $\R^2$, and it is attained when one of the two vectors is the direction of the diagonal of $K$.
\end{proof}

Let us now prove Proposition \ref{prop:PlanarSymmetricCase}:

\begin{proof}[Proof of Proposition \ref{prop:PlanarSymmetricCase}]
We argue like in the proof of Theorem \ref{thm_Lambda_n_1}. For any $K\in\mathcal{K}_0^2$ with $|K|=1$, if $w_1,w_2$ are the principal axes of the inertia ellipsoid $Z_2(K)$ of $K$,  and taking into account that $L_{2,0}=\frac{1}{\sqrt{12}}$, we have 
\begin{equation}\label{eq:ChainOfEquations}
\tilde{\Lambda}(K)\leqslant\frac{|K|}{|K\cap{w_1^\perp}||K\cap{w_2^\perp}|}\leqslant12L_K^2\leqslant 1.
\end{equation}
Besides, by Lemma \ref{lem:PlanarBoxes}, we have that
$$
\tilde{\Lambda}_0(2)\geqslant\lim_{l\to\infty}\frac{l^4}{l^4+1}=1.
$$
Therefore, $\tilde{\Lambda}_0(2)=1$. If there exists a convex body $K$ (we can assume that $|K|=1$) such that $\tilde{\Lambda}(K)=1$ then for such $K$ all the inequalities in \eqref{eq:ChainOfEquations} are equalities. In particular, if we have equality in the second inequality, $K$ is cylindrical both with respect to $w_1$ and $w_2$, which implies that $K$ is a box. But in this case, the Lemma \ref{lem:PlanarBoxes} gives $\tilde{\Lambda}(K)<1$.
\end{proof}

\begin{rmk}
In \cite{FHL}, the authors claimed that if $K\in\mathcal{K}_0^2$, then $\tilde{\Lambda}(K)=1$ if and only if $K$ is a parallelogram with one of its diagonals perpendicular to the edges. The following example shows that such characterization was not correct. Let
$$
K=\textrm{conv}\left\{\left(0,\frac{1}{2}\right),\left(1,\frac{1}{2}\right),\left(0,-\frac{1}{2}\right),\left(-1,-\frac{1}{2}\right)\right\},
$$
which is a symmetric parallelogram with the diagonal from $\left(0,\frac{1}{2}\right)$ to $\left(0,-\frac{1}{2}\right)$ perpendicular to the edge from $\left(0,\frac{1}{2}\right)$ to $\left(1,\frac{1}{2}\right)$. Notice that $|K|=1$ and if we take $w_1$ in the direction of the diagonal from $\left(1,\frac{1}{2}\right)$ to $\left(-1,-\frac{1}{2}\right)$, we have that $w_1^\perp$ intersects the boundary of $K$ at the points $P=\left(-\frac{1}{6},\frac{1}{3}\right)$ and $-P$ and then, taking $w_2$ orthogonal to $w_1$ we have that
$$
\frac{|K|}{|K\cap{w_1^\perp}||K\cap{w_2^\perp}|}=\frac{3}{5}<1.
$$
Thus, it is not true that $\tilde{\Lambda}(K)=1$.
\end{rmk}
\section{Restricted Versions}\label{sec:RestrictedVersions}

In this section we will prove reverse versions of restricted Loomis-Whitney and restricted dual Loomis-Whitney inequalities. We start proving Theorem \ref{thm:ReverseLocalLoomisWhitney}.

\begin{proof}[Proof of Theorem \ref{thm:ReverseLocalLoomisWhitney}]
	Let $K\in\mathcal{K}^n$, $H\in G_{n,d}$ and let $\Pi^*K$ be the polar projection body of $K$. Since $\Pi^*K$ is a centrally symmetric convex body, $\Pi^*K\cap H$ is a centrally symmetric convex body in $H$, using \cite[Lemma 5.5]{CGG}, there exists a rectangular cross-polytope $C$ contained in $\Pi^*K\cap H$ such that
	$$
	|\Pi^*K\cap H|\leqslant d!|C|.
	$$
	That is, there exist  $d$ orthogonal vectors $\{w_i\}_{i=1}^d\in S^{n-1}\cap H$  such that
	$$
	C={\rm conv}\{\pm\rho_{\Pi^*K}(w_i)w_i\}_{i=1}^d\subseteq\Pi^*K\cap H
	$$
	and
	$$
	|\Pi^*K\cap H|\leqslant d!|C|=\prod_{i=1}^d2\rho_{\Pi^*K}(w_i)=\frac{2^d}{\prod_{i=1}^d|P_{w_i^\perp}K|}.
	$$
	Since, by \eqref{eq:SectionsPolarProjectionBody}, we have
	$$
	|\Pi^*K\cap H|\geqslant\frac{{{n+d}\choose{n}}}{n^d|K|^{d-1}|P_{H^\perp}K|},
	$$
	we obtain
	$$
	|P_{H^\perp}K||K|^{d-1}\geqslant\frac{{{n+d}\choose {n}}}{(2n)^d}\prod_{i=1}^d|P_{w_i^\perp}K|.
	$$
\end{proof}


Let us now prove the restricted dual Loomis-Whitney inequality given in Theorem \ref{thm:ReverseLocalDualLoomis-Whitney}.
\begin{proof}[Proof of Theorem \ref{thm:ReverseLocalDualLoomis-Whitney}]
	Let $K\in\mathcal{K}_c^n$ be a centered convex body. We can assume, without loss of generality, that $|K|=1$. If $H\in G_{n,d}$, by the reverse Loomis-Whitney inequality \eqref{eq:ReverseLoomisWhitneyCGG} applied to the convex body $P_H(Z_d(K))$, with the value of the constant estimated in \cite{KSZ}, there exists an absolute constant $c$ and an orthonormal basis $\{w_j\}_{j=1}^d$ of $H$ such that
	$$
	|P_H(Z_d(K))|^{d-1}\geqslant\frac{1}{(cd)^\frac{d}{2}}\prod_{j=1}^d|P_{H\cap w_j^\perp}Z_d(K)|.
	$$
	Using \eqref{eq:ProjectionsCentroidBodies}, we get that there exist two absolute constants $c_1,c_2$ such that
	$$
	c_1^d\leqslant|K\cap H^\perp||P_{H}Z_d(K)|\leqslant c_2^d.
	$$
Therefore, for every $1\leqslant j\leqslant d$
	$$
	c_1^{d-1}\leqslant|K\cap (H^\perp\oplus\langle w_j\rangle)||P_{H\cap w_j^\perp}Z_{d-1}(K)|\leqslant c_2^{d-1}.
	$$
Combining the above with the fact that $Z_{d-1}(K)\subseteq Z_d(K)$, it follows that
	$$
	\frac{c_2^{d(d-1)}}{|K\cap H^\perp|^{d-1}}\geqslant\frac{c_1^{d(d-1)}}{(cd)^\frac{d}{2}}\frac{1}{\prod_{j=1}^d|K\cap (H^\perp\oplus\langle w_j\rangle)|},
	$$
which gives the result.
\end{proof}

\bibliographystyle{plain}
\bibliography{biblio}
\end{document}